\documentclass{amsart} 
\usepackage{amssymb, enumerate}

\usepackage[mathscr]{euscript}
\usepackage{extpfeil}

\usepackage{tikz-cd}
\usepackage{xcolor}

\usepackage{tikz-cd}

\usepackage{hyperref}
\hypersetup{%
  bookmarksnumbered=true,%
  colorlinks=true,%
  linkcolor=blue,%
  citecolor=blue,%
  filecolor=blue,%
  menucolor=blue,%
  urlcolor=blue,%
  bookmarksopen=true,%
  bookmarksdepth=2,%
  pageanchor=true}

\usepackage{mathtools}
\usepackage{todonotes}

\numberwithin{equation}{section}

\theoremstyle{plain}
\newtheorem{theorem}[equation]{Theorem}
\newtheorem{proposition}[equation]{Proposition}
\newtheorem{lemma}[equation]{Lemma} 
\newtheorem{corollary}[equation]{Corollary}

\theoremstyle{definition}
\newtheorem{definition}[equation]{Definition}
\newtheorem{example}[equation]{Example}

\theoremstyle{remark}
\newtheorem{remark}[equation]{Remark} 
\newtheorem*{ack}{Acknowledgements}

\newcommand{\agr}{\operatorname{gr}}

\newcommand{\bs}[1]{\boldsymbol{#1}}

\newcommand{\dcoh}[1]{\operatorname{D}^b(\operatorname{coh} #1)}

\newcommand{\fm}{\mathfrak{m}}

\newcommand{\hh}{\operatorname{H}}

\newcommand{\length}{\operatorname{length}}
\newcommand{\lotimes}{\otimes^{\mathbf{L}}}

\newcommand{\minspec}{\operatorname{minSpec}}

\newcommand{\proj}{\operatorname{Proj}}

\newcommand{\rank}{\operatorname{rank}}
\newcommand{\rees}{\mathscr{R}}
\newcommand{\spec}{\operatorname{Spec}}
\newcommand{\thick}{\operatorname{thick}}

 \newcommand{\bbp}{\mathbb{P}}
 \newcommand{\bbQ}{\mathbb{Q}}
 \newcommand{\bbZ}{\mathbb{Z}}

\newcommand{\mbf}[1]{\mathbf{#1}}
\newcommand{\mcal}[1]{\mathcal{#1}}
\newcommand{\mcend}{\mcal End}
\newcommand{\mce}{\mcal E}
\newcommand{\mcf}{\mcal F}
\newcommand{\mcg}{\mcal G}
\newcommand{\mch}{\mcal H}
\newcommand{\mci}{\mcal I}
\newcommand{\mck}{\mcal K}
\newcommand{\mcl}{\mcal L}
\newcommand{\mco}{\mcal O}
\newcommand{\mcu}{\mcal U}

 \newcommand{\chr}{\operatorname{char}}
  \newcommand{\fp}{\mathfrak{p}}
  \newcommand{\fq}{\mathfrak{q}}

\begin{document}

\title[Ulrich modules]{Ulrich modules over \\ local rings of dimension two}

\author[Iyengar]{Srikanth B. Iyengar}
\address{
Department of Mathematics\\
University of Utah\\ 
Salt Lake City, UT 84112\\ 
U.S.A.}

\author[Ma]{Linquan Ma}
\address{Department of Mathematics\\
Purdue University\\
West Lafayette, IN 47907\\
U.S.A.}

\author[Walker]{Mark E.~Walker}
\address{Department of Mathematics\\
University of Nebraska\\
Lincoln, NE 68588\\
U.S.A. }

\begin{abstract}
      It is proved that Ulrich modules exist for a large class of local rings of dimension two. This complements earlier work of the authors and Ziquan Zhuang that described complete intersection domains of dimension two that admit no Ulrich modules. As an application, it is proved that, for this class of rings, the length of a nonzero module of finite projective dimension is at least the multiplicity of the local ring.
     \end{abstract}

\date{\today}



\maketitle

\section{Introduction}
This paper establishes the existence of Ulrich modules for a large class of local rings of dimension two.  Throughout $(R, \fm, k)$ denotes a local noetherian ring $R$ with maximal ideal $\fm$ and residue field $k = R/\fm$; in this introduction we assume $k$ is infinite. Let $B = \proj \left(\bigoplus \fm^n\right)$ denote the blow-up of $\spec R$ at its closed
point, $p: B \to \spec R$ the canonical projection, and $E = \proj \left( \bigoplus \fm^n/\fm^{n+1} \right)$ the fiber of $p$ over $\{\fm\}$.

\begin{theorem} 
\label{ithm:existence} In the setup above, $R$ has an Ulrich module if and only if there exists an
  Ulrich sheaf $\mcu$ on $E$ that extends to a coherent sheaf $\mcf$ on $B$; in this case, $p_* \mcf$ is an
  Ulrich $R$-module.
\end{theorem}

By $\mcu$ extends to $\mcf$ we mean that the former is the derived pullback of the later; see Section~\ref{se:Ulrich-things}.  It was previously known that an Ulrich $R$-module gives rise to an Ulrich sheaf on $E$ that extends to $B$.  Using this fact, in our joint work with Zhuang~\cite{Iyengar/Ma/Walker/Zhuang:2024}, we constructed examples of two dimensional rings that admit no Ulrich modules. The main point of the current paper is to use the converse to establish a class of two dimensional rings for which Ulrich modules do exist:

\begin{theorem}
\label{ithm:dim2}
Let $R$ be a complete local ring that is equidimensional of dimension two. If the $k$-scheme $E$ is geometrically reduced, 
then $R$ admits an Ulrich module that is locally free of constant rank on the punctured spectrum $\spec R \setminus \{\fm\}$. 
\end{theorem}

We establish this theorem by proving that any $E$ as above admits an Ulrich \emph{bundle}, and verifying that any bundle of $E$ extends to $B$, so Theorem~\ref{ithm:existence} applies.

 In Theorem~\ref{ithm:dim2}, the assumption that $E$ is reduced is necessary: the counter-examples in \cite{Iyengar/Ma/Walker/Zhuang:2024} to the existence of Ulrich modules arise from situations where $E$ is an infinitesimal thickening of a smooth curve. The proof of Theorem~\ref{ithm:dim2}, sketched in the previous paragraph, shows that, in fact, $E$ admits no Ulrich bundles.     

 One consequence of preceding theorem is that the Length Conjecture from \cite{Iyengar/Ma/Walker:2022} holds for  the class of local rings covered by it; see Theorem~\ref{thm:LC}.

\section{Ulrich modules and Ulrich sheaves} 
\label{se:Ulrich-things}
In this section we recall standard facts concerning Ulrich modules and  Ulrich sheaves required in the sequel.

\subsection*{Ulrich modules}
Let $(R, \fm, k)$ be a noetherian local ring, set $d \coloneqq \dim R$ and
let $M$ be a finitely generated $R$-module. The minimal number of
generators of $M$ is written as $\nu_R(M)$
and its multiplicity is written as $e_R(M)$; that is, 
\[
\nu_R(M) = \rank_k(M/\fm M)\quad \text{and}\quad 
e_R(M) = \lim_{n \to \infty} \frac{d!\, \cdot \length_R(M/\fm^nM)}{n^{d}}.
\]
When $M$ is a maximal Cohen-Macaulay, we have $\nu_R(M)\le e_R(M)$; see, for instance, \cite[Section~3]{Ulrich:1984}. An
$R$-module $M$ is an \emph{Ulrich module} if it is nonzero, maximal Cohen-Macaulay, and satisfies $\nu_R(M)= e_R(M)$.

When the field $k$ is infinite one can find a system of parameters $\bs x\coloneqq x_1, \dots, x_d$ for $R$ that
generate a reduction of $\fm$; that is to say, $(\bs x) \fm^j = \fm^{j+1}$ for $j \gg 0$. The ideal $(\bs x)$ is said to
be a \emph{minimal reduction} of $\fm$. In this case, a non-zero finitely generated $R$-module $M$ is Ulrich if and only
if it is maximal Cohen-Macaulay and $(\bs x) M = \fm M$.

It is convenient to extend the notion of Ulrich modules to complexes: We say that an $R$-complex $F$ is an \emph{Ulrich module} provided $\hh_i(F) = 0$ for all $i \ne 0$ and the $R$-module $\hh_0(F)$ is Ulrich. Similarly, $F$ is a \emph{$k$-vector space} provided $\hh_i(F) \ne 0$ for all $i \ne 0$ and $\hh_0(F)$ is annihilated by the maximal ideal of $R$.

Let $K$ be the Koszul complex on a system of parameters generating a minimal reduction of $\fm$. 
A non-zero, finitely generated $R$-module $M$ is Ulrich if and only if $K \otimes_R M$ is a $k$-vector space.
This fact generalizes to complexes:

\begin{lemma} 
\label{le:Ucomplex} 
Let $R$ be a local ring,  $K$ the Koszul complex on sequence of elements that generate a minimal reduction of the
maximal ideal $\fm$ of $R$, and $F$ an $R$-complex that is not exact and with $\hh_i(F)$ finitely generated for each $i\in\bbZ$. Then $F$ is
an Ulrich module if and only if $K \otimes_R F$ is a 
$k$-vector space.
\end{lemma}

\begin{proof} Let $J = (x_1, \dots, x_d)$ be the given minimal reduction.
If $F$ is Ulrich, then 
\[
K \otimes_R F \simeq K \otimes_R \hh_0(F) \simeq \hh_0(F)/J \hh_0(F) = \hh_0(F)/\fm \hh_0(F),
\]
a $k$-vector space. 

Conversely,  assume $K \otimes_R F$ is a $k$-vector space. Let $K(j)$ denote the Koszul complex on $x_1, \dots, x_j$ for $0 \leq j \leq d$. For each $0 < j \leq d$ the standard mapping cone exact sequence
\[
0\longrightarrow K(j-1) \longrightarrow K(j) \longrightarrow \Sigma K(j-1)\longrightarrow 0
\]
gives rise to an exact sequence of $R$-modules
\[
\cdots \to \hh_i(K({j-1}) \otimes_R F) \xrightarrow{\pm x_j} \hh_i(K({j-1}) \otimes_R F) \to
\hh_i(K(j) \otimes_R F) \to \cdots 
\]
Since $K \otimes_R F = K(d) \otimes_F F$ only has homology in degree $0$,  by descending induction on $j$ and Nakayama's Lemma, we deduce that $K(j) \otimes_R F$  has homology only in degree $0$ for all $0 \leq j \leq d$.  The case $j = 0$ gives an isomorphism in the derived category $F \cong \hh_0(F)$ and thus $K \otimes_R \hh_0(F)$ is a $k$-vector space.
As noted above, this is equivalent to $\hh_0(F)$ being an Ulrich module.
\end{proof}

\subsection*{Projective schemes}
Let $k$ be a field.  A \emph{projective $k$-scheme} is a pair $(X, \mcl)$ where $X$ is a $k$-scheme and $\mcl$ is a very ample line bundle on $X$ relative to $k$; that is to say, there is a closed immersion of $k$-schemes  $i\colon X \hookrightarrow \bbp^n_k$ for some $n$ such that  $\mcl \cong i^* \mco_{\bbp^n_k}(1)$. When there is no danger of confusion, we write $\mcl$ as $\mco_X(1)$ or even just $\mco(1)$.  Given a quasi-coherent sheaf $\mcf$ of $\mco_X$-modules, set
\[
\mcf(j)\coloneqq \mcf \otimes_{\mco_X} \mcl^{\otimes j}\,.
\]
By $\bbp^d_k$ we mean the projective $k$-scheme $(\bbp^d_k, \mco_{\bbp^d}(1))$. 

A \emph{linear Noether normalization} of a projective $k$-scheme $(X,\mcl)$ of dimension $d$
is a dominant morphism of $k$-schemes $f\colon X \to \bbp^{d}_k$ such that $\mcl \cong f^* \mco_{\bbp^d_k}(1) $; such a thing exists if $k$ is infinite.

\subsection*{Ulrich sheaves}
Let $(X, \mco_X(1))$ be a projective $k$-scheme and set $d = \dim X$.

An \emph{Ulrich sheaf} on $X$ is a nonzero coherent sheaf $\mcu$  such that, for each $t$ in $[-d,-1]$, the sheaf $\mcu(t)$ has no cohomology, that is to say, $\hh^i(X, \mcu(t)) = 0$ for $i$.

When $d = 0$ every nonzero coherent sheaf is Ulrich and when $d=1$ existence an of Ulrich sheaf is tantamount to existence of a sheaf with no cohomology.

When $d \geq 1$ an Ulrich sheaf $\mcu$ on $X$ also satisfies $\hh^0(X, \mcu(t)) = 0$ for all $t < 0$ and $\hh^{d}(X, \mcu(t)) = 0$ for all $t \geq -d$. This follows from the definition and the Koszul exact sequence
\[
0 \longrightarrow \mcu(t-d-1)  \longrightarrow  \cdots  \longrightarrow  \mcu(t-2)^{\binom {d+1}2}  \longrightarrow  
    \mcu(t-1)^{d+1}  \longrightarrow \mcu(t)  \longrightarrow 0.
\]
See, for instance, \cite[Proposition~2.1]{Eisenbud/Schreyer:2003} for details. 

A coherent sheaf $\mcu$ on $\bbp^d_k$ is Ulrich if and only if it is isomorphic to $\mco_{\bbp^d}^r$ for some $r \geq 1$; see \cite[Proposition~2.1]{Eisenbud/Schreyer:2003}. By the projection formula this generalizes to the following statement. 

\begin{lemma} 
Let $X$ be a projective $k$-scheme and $f\colon X\to\bbp^d_k$ a linear Noether normalization. A coherent sheaf $\mcf$ on $X$ is Ulrich if and only if $f_* \mcf \cong \mco_{\bbp^d}^r$ for some $r \geq 1$. \qed
\end{lemma}

As in the affine case, we say that a complex $\mcg$ of quasi-coherent sheaves over $X$ is \emph{Ulrich} if $\hh^i(\mcg)=0$ for $i\ne 0$ and the sheaf $\hh^0(\mcg)$ is a Ulrich.

\subsection*{Blowups}
Let $(R, \fm, k)$ be a local ring. Let $B$ be the blow-up of $\spec R$ at $\{\fm\}$, let $p\colon B \to \spec R$ be the
structure map, and $E$ the fiber of $p$ over $\{\fm\}$. Thus $B = \proj \rees_\fm(R)$ and $E = \proj \agr_\fm(R)$,
where
\[
\rees_{\fm}(R) \coloneqq \bigoplus_{n\geqslant 0} \fm^n \quad\text{and}\quad
\agr_\fm(R)\coloneqq R/\fm \otimes_R \rees_\fm(R) = \bigoplus_{n \geqslant 0} \fm^n/\fm^{n+1},
\]
and there is a cartesian square
\begin{equation}
    \label{eq:cartesian}
\begin{tikzcd}
E \arrow[d,"q" swap] \arrow[r, hookrightarrow,"j"] & B \arrow[d,"p"] \\
\spec k \arrow[r, hookrightarrow] &\spec R
\end{tikzcd}
\end{equation}
where the maps are the obvious ones. The ideal $\mci_E$ cutting out $E$ as a subscheme of $B$ is isomorphic to $\mco_B(1)$ via a map we write as $\sigma \colon \mco_B(1) \xrightarrow{\cong} \mci_E$; that is to say, there is an exact sequence 
\begin{equation}
\label{eq:blowup-ses}
0 \longrightarrow \mco_B(1) \xrightarrow{\ \sigma\ } \mco_B  \longrightarrow j_* \mco_E  \longrightarrow 0\,.    
\end{equation}

\subsection*{Extensions}
With the notation as above,  given a coherent sheaf $\mcu$ on $E$, we  say a coherent sheaf $\mcf$ on $B$  is an \emph{extension} of $\mcu$, or that  $\mcu$ \emph{extends to} $\mcf$, if
\[
\mcu \cong j^* \mcf \quad\text{and}\quad \mbf{L}^i j^* \mcf = 0 \text{ for all $i  >0$}\,;
\] 
equivalently, there is a quasi-isomorphism $\mcu \simeq \mbf{L} j^* \mcf$ of complexes of coherent sheaves on $E$. The exact sequence \eqref{eq:blowup-ses} gives that $ \mbf{L}^i j^* \mcf = 0$ for all $i \notin \{0,1\}$ and $j_* \mbf{L}^1 j^*\mcf \cong \ker(\mcf(1) \xrightarrow{\sigma} \mcf)$. So, $\mcf$  is an extension of $\mcu$ if and only if $j^* \mcf \cong \mcu$ and
$\mcf(1) \xrightarrow{\sigma} \mcf$ is injective.

The following criterion can be used in some situations to construct extensions:

\begin{proposition}
    \label{pr:lift-obstruction}
Assume $R$ is  complete. Let $\mce$ be a vector bundle  on $E$ and $\mcend_{\mco_E}(\mce)$ its endomorphism sheaf. If $\hh^2(E, \mcend_{\mco_E}(\mce)(i)) = 0$ for all $i \geq 1$, then there is a vector bundle $\mcf$ on $B$ such that $\mce \simeq  \mbf{L} j^*(\mcf)$.
\end{proposition} 

\begin{proof}
Recall $E$ is the fiber of $p\colon B \to \spec R$ over $\spec(R/\fm)$. Write $E_j$ for the fiber of $p$ over $\spec(R/\fm^{j+1})$, so that  $E = E_0$  and each $E_j$ is an infinitesimal thickening of $E$. By \cite[III.7.1]{SGA1:2003},  for each $j \geq 0$, given a vector bundle $\mce_j$ on $E_j$ that is an extension of $\mce$, the obstruction to extending it to a vector bundle on $E_{j+1}$ is an element of $\hh^2(E, \mcend_{\mco_E}(\mce) \otimes_{\mco_E} \mci_E^{j+1}/\mci_E^{j+2})$. Since $\mci_E \cong \mco_B(1)$,  we have $\mci_E^{j+1}/\mci_E^{j+2} \cong \mco_E(j+1)$. So, under our assumptions, these obstructions all vanish, and we may construct a sequence $\mce = \mce_0, \mce_1, \mce_2, \dots$ of vector bundles on $E = E_0, E_1, E_2, \dots$ such that the restriction of $\mce_j$ to $E_{j-1}$ is isomorphic to $\mce_{j-1}$ for all $j \geq 1$. By \cite[10.11.1]{EGA1:1961} such a sequence determines a coherent sheaf on the formal scheme $\hat{B}$ given by completing $B$ along $E$.

Since $R$ is complete, Grothendieck's Existence Theorem~\cite[5.1.6]{EGA3:1961}  gives that the functor $\mcf \mapsto \hat{\mcf}$ induces an equivalence of categories between coherent sheaves on $B$ and those on the formal scheme $\hat{B}$. This proves that there exists a coherent sheaf $\mcf$ on $B$ whose restriction to $E_j$ is isomorphic to $\mce_j$ for all $j \geq 0$, via isomorphisms compatible with those in the sequence above; see also \cite[5.1.7]{EGA3:1961}.

Finally, we show $\mcf$ is locally free:\footnote{We thank Luc Illusie for pointing out this proof to us.}  Let $y \in B$ be any point lying in the closed fiber of $p$. Then $N \coloneqq \mcf_y$ is a finitely generated module over the $R$-algebra $A = \mco_{B,y}$ and the completion of $N$ at $\fm \cdot A \subseteq \fm_A$ is a free $\hat{A}_{\fm A}$-module. It follows that $N$ is a free $A$-module. Thus,  the set of points of $B$ at which $\mcf$  is not locally free forms a closed subset that does not meet the closed fiber. Since $p$ is proper, this set must be
empty. 
\end{proof}

\begin{corollary}
    \label{cor:no-obstruction}
    When $R$ is a  complete local ring of dimension two, any vector bundle  on $E$  extends to a vector bundle on $B$. \qed
\end{corollary}

\section{Ulrich modules from Ulrich sheaves}
\label{se:UM-from-US}

We keep the notation from the previous section: $(R,\fm,k)$ is a local ring and $p\colon B \to \spec R$ is the blowup of $\spec R$ at $\{\fm\}$, with exceptional fiber $E$; as before $j\colon E\to B$ is the canonical inclusion. Given a finitely generated $R$-module $M$, its {\em strict transform} is the coherent sheaf on $B$ associated to  the graded $\rees_\fm(R)$-module $\rees_\fm(M) = \bigoplus_i \fm^i M$.  The result below is a more precise version of Theorem~\ref{ithm:existence}. Part (1) is \cite[Lemma~2.2]{Iyengar/Ma/Walker/Zhuang:2024}, and included here for ease of reference.

\begin{theorem}  
\label{thm:lift} Let $R$ be a local ring with infinite residue field $k$. The following statements hold.
\begin{enumerate}[\quad\rm(1)]
    \item 
If $M$ is an Ulrich $R$-module, then the coherent sheaf  on $E$ associated to the graded module $\agr_\fm M$ is an Ulrich sheaf on $E$ that extends to the strict transform of $M$ on $B$. 
    \item 
If $\mcu$ is an Ulrich sheaf on $E$ and $\mcf$ is an extension of $\mcu$ to $B$,
then $\mbf{R}p_* \mcf$  is an Ulrich $R$-module.    
\end{enumerate}
\end{theorem}

It may help to illustrate the result above in an example.

\begin{example}
 Suppose $\dim E = 0$, so that every nonzero coherent sheaf on $E$ is Ulrich. Then $B = \spec A$ is an affine scheme and $\fm A = (s)$ where $s$ is not a zerodivisor.  Part (2) of the theorem asserts that if $M$ is finitely generated and $s$-torsion free $A$-module, then it is Ulrich as an $R$-module.  
 
 Indeed, since $k$ is infinite there exists an element $x \in \fm$ with $x \fm^j = \fm^{j+1}$ for $j \gg 0$. Hence $x (\fm A)^j = (\fm A)^{j+1}$, that is to say, $(xs^j) = (s^{j+1})$ for $j \gg 0$.  Since $s$ is not a zerodivisor it follows that $x = us$ for some unit $u$ of $A$.  In particular, if $M$ is $s$-torsion free, then 
 it is $x$-torsion free and hence maximal Cohen-Macaulay as an $R$-module. Moreover
\[
M/xM = M/sM = M/(\fm A) M = M /\fm M.
\]
Thus $M$ is an Ulrich module. 
\end{example}

The calculation below extracts a key step in the proof of Theorem~\ref{thm:lift}.

\subsection*{A Koszul homology calculation}
Let $B$ be a scheme, $\mci$ a sheaf of ideals that is locally generated by a non-zero-divisor, and $E$ the subscheme of $B$ cut out by $\mci$. Let $\mcf \twoheadrightarrow \mci$ be a surjection with $\mcf$ locally free of rank $d$, and $g\colon \mcf \to \mco_B$ its composition with the inclusion $\mci \hookrightarrow \mco_B$. Let
  \[
  \mck \coloneqq 0 \longrightarrow \Lambda^d \mcf \longrightarrow \cdots 
        \longrightarrow \Lambda^2 \mcf \longrightarrow \mcf \longrightarrow \mco_B \longrightarrow 0
  \]
be the Koszul complex on $g$.

\begin{lemma}
\label{lem:Koszul-calculation}
With the setup as above, let $j\colon E\hookrightarrow B$ the canonical inclusion of schemes. For each integer  $0 \leq n \leq d$ the following statements hold:
\begin{enumerate}[\quad\rm(1)]
    \item 
    There is an isomorphism $\mch_n(\mck) \cong j_* j^* \mch_n(\mck)$.
    \item 
    The sheaf $j^* \mch_n(\mck)$ is locally free of rank $\binom{d-1}{n}$; in particular, $\mch_d(\mck)=0$.
    \item
    There is an exact sequence of locally free $E$-sheaves
 \[
0 \longrightarrow   \mce_{d-n}  \longrightarrow \cdots \longrightarrow \mce_2  
                \longrightarrow \mce_1  \longrightarrow j^* \mch_n(\mck) \longrightarrow 0
\]
with $\mce_i\coloneqq \Lambda^{n+i}(j^*\mcf) \otimes \mathcal{N}^{i}$, where $\mathcal{N}\coloneqq (j^*\mci)^{-1}$ is the normal bundle of $j$.
\end{enumerate}  
\end{lemma}

\begin{proof}
Since $\mck$ is a sheaf of differential graded algebras, the action of $\mco_B$ on $\mch_n(\mck)$ factors though the canonical map $\mco_B \to \mch_0(\mck)$ and hence we have canonical isomorphisms $\mch_n(\mck) \cong j_*j^*\mch_n(\mck)$, which justifies (1).

(2)  Let $\mcg$ be the kernel of $g\colon \mcf\to \mco_B$; it is locally free of rank $d-1$.  Since $\mcf/\mcg \cong
  \mci$ is locally free of rank one, for $n\ge 1$ there are canonical exact sequences
  \begin{equation} 
  \label{canex}
  0 \longrightarrow \Lambda^n \mcg \longrightarrow \Lambda^n \mcf \xrightarrow{\ p\ } \Lambda^{n-1} \mcg \otimes \mci \longrightarrow 0,
\end{equation}
and the differential in the Koszul complex $\mck$ is the composition of $p$ with the inclusion
\[
\Lambda^{n-1} \mcg \otimes \mci   \hookrightarrow \Lambda^{n-1} \mcf
\]
induced by the inclusions of $\mci$ into $\mco_B$ and $\mcg$ into $\mcf$.  We may thus identify $\mch_n(\mck)$ with the cokernel of the inclusion
\[
\Lambda^{n} \mcg \otimes \mci   \hookrightarrow \Lambda^{n} \mcg \otimes \mco_B \cong \Lambda^n \mcg
\]
and hence deduce that there are isomorphisms 
\[
  j^*\mch_n(\mck) \cong j^* \Lambda^n \mcg \cong \Lambda^n(j^* \mcg)
\]
for $0 \leq n \leq d$. This justifies (2). 

(3) The Koszul complex associated to the surjection  $\mcf \twoheadrightarrow \mci$ has the form
  \[
  0 \to 
\Lambda^d \mcf \otimes \mci^{1-d} \to  \cdots
\to \Lambda^3 \mcf \otimes \mci^{-2} 
  \to \Lambda^2 \mcf \otimes \mci^{-1} \to \mcf \to \mci \to 0\,,
  \]
which is an exact sequence of locally free sheaves. It follows from the exact sequences \eqref{canex} that for each $0\le n \le d$,  the $n$th syzygy in the exact sequence above is $\Lambda^n\mcg \otimes \mci^{-n+1}$, so we get an exact sequence
  \[
  0 \to
\Lambda^{d} \mcf \otimes \mci^{n-d} \to \cdots \to  \Lambda^{n+2} \mcf \otimes \mci^{-2}  \to \Lambda^{n+1} \mcf \otimes \mci^{-1}  \to \Lambda^n \mcg \to 0 
\]
of locally free $B$-sheaves. Applying $j^*$ yields (3).
\end{proof}

\subsection*{A Koszul complex associated to the blowup}
We prepare to prove Theorem~\ref{thm:lift}. Since $k$ is infinite, $\fm$ has a minimal reduction: elements $\bs x\coloneqq x_0, \dots, x_d$ of $R$ such that the ideal $J=(\bs x)$ satisfies the equality $J \fm^i = \fm^{i+1}$ for all $i \gg 0$. In particular, in the notation of \eqref{eq:cartesian}, the fiber of $p$ over $\spec(R/J)$ is isomorphic to the exceptional fiber of $p$:
\[
\proj(\bigoplus_{n\geqslant 0} \fm^n/J \fm^{n}) \cong E\,.
\]
The means that the images of $\bs x$ under the map $R  \to \Gamma(B, \mco_B)$ (which we also denote $\bs x$) generate the ideal $\mci_E$; that is, we have a surjection $(\bs x)\colon \mco_B^d \twoheadrightarrow \mci_E$ of coherent sheaves.  Let $K$ be the Koszul complex on $\bs x$ and set 
\[
\mck = p^*K = K_{\mco_B}(\bs x)\,;
\]
this is a complex of coherent sheaves on $B$. In the statement below
$\mathrm{Perf}(E)$ is the derived category of perfect complexes over
$E$ and $\thick(S)$ refers to the smallest triangulated subcategory of a
triangulated category that is closed under direct summands and that
contains the objects in $S$. 

\begin{lemma}
\label{le:Hi} 
With the setup as above, for each integer $1 \le n \le d-1$, one has an isomorphism $\mch_n(\mck) \cong j_* j^* \mch_n(\mck)$, the sheaf $j^*\mch_n(\mck)$ is locally free, and as subcategories of $\mathrm{Perf}(E)$ there is an equality
\[
  \thick(j^*\mch_n(\mck),\dots, j^*\mch_{d-1}(\mck))
  =
  \thick(\mco_E(-1),\dots,\mco_E(n-d))\,.
\]
In particular, an object $\mcu$ in $\dcoh{E}$ is an Ulrich sheaf if and only if
\[
  \mbf R q_*(\mcu \otimes j^*\mch_n(\mck)) =0 \quad\text{for all $1\le n \le d-1$.}
\]
\end{lemma}

\begin{proof}
 As noted above, $\mci_E$ is locally generated by a non-zero-divisor and there is a surjection $\mco_B^d \twoheadrightarrow \mci_E$. So applying  Lemma~\ref{lem:Koszul-calculation} with $\mcf = \mco_B^d$, for each $1 \leq n \leq d-1$ we obtain isomorphisms $j_* j^* \mch_n(\mck) \cong \mch_n(\mck)$, which justifies the first part of the statement. Moreover, there is an exact sequence
  \[
  0 \to
\mco_E(n-d) \to \mco_E(n-d+1)^d\to \cdots \to \mco_E(-1)^{\binom d{n+1}} \to j^* \mch_n(\mck) \to 0,.
\]
The asserted equality regarding thick subcategories of  $\mathrm{Perf}(E)$ follows.

The final assertion holds since, for any $\mcu$ in $\dcoh{E}$, the collection 
\[
\{\mcf \in \mathrm{Perf}(E) \mid  \mbf R q_*(\mcu \otimes^{\mbf L}\mcf) = 0 \}
\]
is a thick subcategory of $\mathrm{Perf}(E)$.
\end{proof}

\begin{proof}[Proof of Theorem \ref{thm:lift}] 
As noted earlier, part (1) is \cite[Lemma~2.2]{Iyengar/Ma/Walker/Zhuang:2024}, so we focus on part (2). Recall that $K$ is the Koszul complex on minimal reduction $\fm$ and that $\mck$ is its pull-back to $B$. By the projection formula, we have
\[
\mbf{R} p_* \mcf \otimes K \simeq \mbf{R} p_*(\mcf \otimes \mck),
\]
and so by Lemma \ref{le:Ucomplex} it suffices to prove that $ \mbf{R}p_*(\mcf \otimes \mck)$ is a $k$-vector space.

By assumption, $ \mbf{L}^i j^* \mcf = 0$ for all $i \ne 0$. Moreover, $\mch_n(\mck) \cong j_* j^* \mch_n(\mck)$ with $j^*\mch_n(\mck)$ locally free, by the first part of Lemma \ref{le:Hi}, so there are isomorphisms
\begin{align*}
\mcf \lotimes j_*j^*\mch_p(\mck) 
        &\simeq j_*(\mbf{L}j^*\mcf\lotimes j^*\mch_p(\mck))\\
        &\simeq j_*(\mcu \otimes j^*\mch_p(\mck))
\end{align*}
in the derived category of $B$. Thus the standard spectral sequence 
\[
\operatorname{\mcal{T}or}^{\mco_B}_q(\mcf, \mch_p(\mck)) \Longrightarrow \mch_{p+q}(\mcf\otimes \mck)
\]
collapses to yield an isomorphism  $\mch_n(\mcf \otimes \mck) \cong j_*\left(\mcu \otimes j^*\mch_n(\mck)\right)$.  
By the last part of Lemma \ref{le:Hi} this gives 
\[
\mbf{R}p_*\mch_n(\mcf \otimes \mck) \cong
\begin{cases}
    \iota_* q_* \mcu & n = 0 \\
    0 & \text{otherwise}.
\end{cases}
\]
We thus have an quasi-isomorphism of complexes of $R$-modules
\[
   \mbf{R}p_*(\mcf \otimes \mck) \simeq \iota_* q_*\mcu\,,
\]
and the term on the right is a $k$-vector space. 
\end{proof}

\begin{remark}
The proof gives more, namely, if $\mbf{L}j^*\mcf \simeq j^*\mcf$ and the mapping cone $\mcal C$ of $\mck\to j_*\mco_E$ has the property that
\[
\mbf{R}p_*(\mcf\lotimes \mcal C)\simeq 0
\]
then $\mbf{R}p_*\mcf$ is an Ulrich module.
\end{remark}

Theorem \ref{thm:lift} gives that $R$ admits an Ulrich module provided there is an Ulrich sheaf on the exceptional fiber
$E$ that extends to the blow-up $B$. In general, it is difficult to determine whether a given Ulrich sheaf on $E$
extends, but there is one situation in which it always does. 

\begin{example} 
Suppose $R$ is the localization of a standard graded $k$-algebra $A$ at its homogeneous maximal ideal. In this case we may identify $\agr_\fm(R)$ with $A$ and hence $E$ with $\proj A$. Moreover,  the scheme $B$ is the geometric line bundle over $E$ associated to the invertible  sheaf $\mco_E(1)$, and $j$ is the zero section of this bundle. 

If $\mcu$ is any Ulrich sheaf on $E$, then setting $\mcf = p^* \mcu$, where $p\colon B \to E$ is the structural map for this line bundle, and using that $p \circ j = id_E$,  we see that $\mcf$ is an extension of $\mcu$. In this case $ \mbf{R}p_* \mcf \cong \bigoplus_{t} \hh^0(X, \mcu(t))$, regarded as a graded $A$-module in the standard way. So, Theorem~\ref{thm:lift} recovers \cite[Proposition~2.1]{Eisenbud/Schreyer:2003} that an Ulrich sheaf $\mcu$ on $\proj A$ determines an Ulrich module on $R$, namely, the one obtained by localizing  $\bigoplus_{t} \hh^0(X, \mcu(t))$ at the homogeneous maximal ideal of $A$. 
\end{example}

\section{Ulrich modules on dimension two local rings}
\label{se:dim2}
In this section we establish the existence of Ulrich modules on
certain two-dimensional local rings. Recall that a $k$-scheme is
said to be {\em geometrically reduced} provided $X \times_{\spec(k)} \spec(\overline{k})$ is reduced, where $\overline{k}$ is the algebraic closure of $k$. 

\begin{lemma} 
\label{lem3}
Let $(X, \mco_X(1))$ be projective $k$-scheme of finite type, with $\dim X=1$. If $X$ is geometrically reduced and  equidimensional, then $X$ admits an Ulrich vector bundle;  if in addition $k$ is algebraically closed, $X$ admits an Ulrich line bundle. 
\end{lemma}

\begin{proof}
  Suppose that $k$ is algebraically closed. We need find a line bundle $\mcl$ on $X$ with no cohomology, that is to say, with $\hh^i(X, \mcl) = 0$ for all $i$, for then $\mcl(-1)$ is an Ulrich sheaf on $X$.
  
Let $\mcl$ be a line bundle on $X$ with $\hh^1(X, \mcl) = 0$; for instance, $\mcl = \mco_E(m)$ for $m \gg 0$.
If $\hh^ 0(X, \mcl) = 0$ we are done. If not, choose a nonzero element $f \in \hh^0(X, \mcl)$.  By a \emph{zero} of $f$ we mean a closed point $x$ of $X$  such that lies in the kernel of the canonical map $\hh^0(X, \mcl) \to \hh^0(X, i^* i_* \mcl) \cong k$, where $i\colon \{x\} \to X$ is the inclusion map. 

We claim that $Z(f)$, the set of  zeroes of $f$, is  a proper, Zariski closed subset of the set of all closed points of $X$. 

To see this, let $\{U_i\}$ be an open covering of $X$ by non-empty, affine open subsets on which the bundle $\mcl$ is trivial. For each $U = U_i$ in this collection we have $U = \spec A$ for some reduced one-dimensional noetherian ring $A$. Upon choosing a trivialization of the restriction of $\mcl$ to $U$, the element $f$ is given by an element $g \in A$. The set $Z(f) \cap U$ corresponds to $Z(g)$, the set of maximal ideals of $A$ containing $g$, which is clearly a closed subset of $\spec A$.  This proves $Z(f)$ is closed. Since $f$ is nonzero, there is at least one $i$ such that  the restriction of $f$ to $U = U_i$ is nonzero, and for this $i$ the corresponding $g \in A$ is nonzero. Since $A$ is reduced, $V(g)$ is a proper subset of the set of maximal ideals of $A$. It follows that $Z(f) \ne X$. 

Let $V$ be the set of regular closed points of $X$. Since $X$ is reduced, $V$ is an open, dense subset of $X$ and thus  $V \cap (X \setminus Z(f)) \ne \emptyset$. 
Pick any point $x$ in this set. So, $x$ is a regular closed point of $X$ and  $x$ is not a zero of $f$. 
  
Since $k$ is algebraically closed,  $\hh^0(X, i^*i_*  \mcl)$ is a one-dimensional $k$-vector space, and hence the nonzero map $\hh^0(X, \mcl) \to \hh^0(X, i^* i_* \mcl)$ must
be surjective.  Since $x$ is a regular point, and the local ring $\mco_{X,x}$ is of dimension 1 (this is where one needs that $X$ is equidimensional) the kernel $\mcl'$ of the canonical map $\mcl \twoheadrightarrow  i_* i^* \mcl$ is also a line bundle. The  exact sequence in cohomology
\[
  0 \to \hh^0(X, \mcl') \to \hh^0(X, \mcl) \to \hh^0(X, i_* i^* \mcl)  \to \hh^1(X, \mcl') \to \hh^1(X, \mcl) \to 0
\]
thus  shows that $\hh^1(X, \mcl') = 0$ and $\dim_k \hh^0(X, \mcl') = \dim_k \hh^0(X, \mcl) - 1$. Continuing in this fashion we arrive at a bundle with no cohomology, as desired.

Suppose $k$ is general field.  Set
\[
X_{\overline{k}} \coloneqq  X \times_{\spec(k)} \spec(\overline{k})\,;
\]
this is reduced, by hypothesis. Since $X$ is equidimensional and the map $X_{\overline k}\to X$ is flat and integral, $X_{\overline k}$ is
equidimensional as well.  Hence $X_{\overline k}$ admits an Ulrich line bundle $\mcl$, as has been proved already. This bundle is extended
from a line bundle $\mcl'$ on $X_{\ell}$ for some finite extension $\ell$ of $k$ contained in $\overline{k}$, in the sense that
$\mcl = p^* \mcl'$ where $p\colon X_{\overline{k}} \to X_\ell$ is the canonical map. Since
\[
\hh^*(X_\ell, \mcl'(-1)) \otimes_\ell \overline{k} \cong \hh^*(X_{\overline{k}}, \mcl(-1)),
\]
it follows that $\mcl'$    is an Ulrich line bundle on $X_\ell$. 
   
The canonical map $q\colon X_\ell \to X$ is finite and linear, that is to say, $q^* \mco_{X}(1) \cong \mco_{X_\ell}(1)$, and hence $q_* \mcl'$ is a Ulrich sheaf on $X$. Moreover,
it is locally free  with constant rank equal to the degree of the field extension $\ell/k$. 
\end{proof}

\begin{remark}
The hypotheses of Lemma \ref{lem3} of imply $X$ is Cohen-Macaulay. When $X$ is not Cohen-Macaulay, there cannot be any Ulrich bundles on $X$.
\end{remark}        

\begin{theorem}  
\label{thm:curves}
Let $(R, \fm ,k)$ be a complete local ring of dimension two with $k$ infinite. Assume $R$ is equidimensional and that $\proj (\agr_\fm(R))$ is geometrically reduced. Then $R$ admits an Ulrich module $U$ that is locally free of constant rank $r > 0$ on the punctured spectrum $\spec R \setminus \{\fm\}$; if $k$ is algebraically closed, then such a module exits with $r = 1$. 
\end{theorem}

\begin{proof}
Set $E = \proj (\agr_\fm(R))$. The ring $R$, being complete, is universally catenary, and since it is assumed to be equidimensional, $E$ is also equidimensional; see, for instance,
  \cite[Proposition~5.4.8]{Huneke/Swanson:2006}.  By assumption $E$ is geometrically reduced, and Lemma \ref{lem3} applies to give an Ulrich
  vector bundle $\mce$ on $E$ of constant rank $r > 0$; when $k$ is algebraically closed we may take $r=1$.  Since $\dim E = 1$, by
  Corollary \ref{cor:no-obstruction} there is a coherent sheaf $\mcf$ on $B$ that is locally free of rank $r$ and such that
  $\mce \cong \mbf{L} j^* \mcf$.  By Theorem \ref{thm:lift}, $U\coloneqq p_* \mcf$ is an Ulrich module on $R$.  The coherent sheaf on
  $\spec(R) \setminus \{\fm\}$ given as the restriction of $U$ coincides with the restriction of $\mcf$ to $B \setminus E$ under the
  canonical isomorphism $B \setminus E \cong \spec(R) \setminus
  \{\fm\}$. So, since $\mcf$ is locally free of constant rank $r$ on
  $B$, the module $U$ is locally free of constant rank $r$ on the punctured spectrum.
\end{proof}

\section{Modules of finite projective dimension}
\label{sec:application}
In this section we give an application of the existence of Ulrich modules for certain two-dimensional local rings to the \emph {Length Conjecture} stated in \cite{Iyengar/Ma/Walker:2022}. First, some notation: We write $G_0(R)$ for the Grothendieck of $R$; that is, $G_0(R)$ is the abelian group generated by isomorphism classes of finitely generated $R$-modules modulo relations coming from short exact sequences of such. Set $G_0(R)_\bbQ = G_0(R) \otimes_\bbZ \bbQ$.

Set $\chi(N) = \sum_i \length_R H_i(N)$ for a bounded complex $N$ with finite length homology,  and for a finite free complex $F$ having finite length homology and a finitely generated module $M$, set 
\[
\chi(F, M) = \chi(F \otimes_Q M)\,.
\]
For a fixed $F$, $\chi(F, -)$ is additive on short exact sequences and thus we may extend its defintion to
$\chi(F, \alpha) \in \bbQ$  for any $\alpha \in G_0(R)_\bbQ$.
We write $E(R)_\bbQ$ for the quotient of $G_0(R)_\bbQ$ modulo classes that are numerically equivalent to zero; that is to
say, modulo classes $\alpha \in G_0(R)_\bbQ$ such that $\chi(F, \alpha) = 0$ for all finite free complexes $F$ having finite length
homology.

\begin{proposition}  \label{prop:E}
Let $R$ be local ring that is a homomorphic image of a regular local ring and that is equidimensional of dimension two. The map sending a finitely generated $R$-module $M$ to the tuple $(length_{R_\fp}(M_\fp))_{\fp \in \minspec R}$,
where $\minspec R$ denotes the set of minimal primes of $R$, induces an isomorphism
\[
  E(R)_\bbQ \xrightarrow{\cong} \bigoplus_{\fp \in \minspec R} \bbQ.
  \]
\end{proposition}

\begin{proof} 
This follows from~\cite[Proposition~3.7]{Kurano:2004};
see also Proposition \ref{prop:wd} below. 
\end{proof}

\begin{theorem}  
\label{thm:LC}
Let $(R, \fm, k)$ be a local ring dimension two  such that its $\fm$-adic completion is equidimensional and the projective $k$-scheme $\proj(\agr_\fm(R))$ is geometrically reduced. For any finite free $R$-complex $F= 0 \to F_2 \to F_1 \to F_0 \to 0$ that is minimal and has non-zero, finite length homology, one has
\[
\chi(F) \geq e(R)\cdot \max\left\{\rank F_0,  \frac{\rank F_1 }{2}, \rank F_2 \right\}.
\]
\end{theorem}

\begin{proof}
By passing to the completion of $R$ we may assume it is complete. There exists a faithfully flat local and integral extension $(R, \fm, k) \subseteq (R', \fm', k')$ such $\fm R' = \fm'$ and $k'$ is an infinite algebraic field extension of $k$. Then 
\[
\proj(\agr_{\fm'} R') = \proj(\agr_\fm(R)) \times_{\spec k} \spec k'\,,
\]
and hence it is also geometrically reduced. Since $\chi^R(F) = \chi^{R'}(R' \otimes_R F)$, passing to $R'$ we may assume also that the residue field of $R$ is infinite.

By  Theorem~\ref{thm:curves}, $R$ admits an Ulrich module $U$ such that for some $r > 0$, the $R_\fp$-module $U_\fp$ is free of rank $r$ for each minimal prime $\fp$. By Proposition~\ref{prop:E}, one has $[U] = r [R]$ in $E(R)_\bbQ$, and from the formula
\[
e_R(M) = \sum_{\fp \in \minspec R} \length_{R_\fp}(M_\fp) e(R/\fp)\,,
\]
 we have $e_R(U) = r \cdot e(R)$.  Combining these facts gives
\begin{equation} \label{eq:special}
\chi(F  \otimes_R U) = r\cdot \chi(F) = \frac{e_R(U) \chi(F)}{e(R)}.
\end{equation}
An argument in \cite{Iyengar/Ma/Walker:2022} completes the proof: Let $K$ be the Koszul complex on a minimal reduction of the maximal ideal. Since $U$ is maximum Cohen-Macaulay, we have a quasi-isomorphism  $U \otimes_R F \simeq \hh_0(U \otimes_R F)$ and since $U$ is  Ulrich we have a quasi-isomorphism $K \otimes_R U \simeq k^{e_R(U)}$. Thus
\[
K \otimes_R \hh_0(U \otimes_R F) \simeq k^{e_R(U)} \otimes_R F,
\]
and, since $F$ is minimal and  $\rank_R(K_i) = \binom 2i$, this gives
\[
\binom{2}{i} \length_R \hh_0(U \otimes_R F) \geq   e_R(U) \rank_R(F_i),
\, \text{ for $0 \leq i \leq 2$.} 
\]
On the other hand, from \eqref{eq:special} we have
\[
\length_R \hh_0(U \otimes_R F) =
\chi(F \otimes_R U) =\frac{e_R(U) \chi(F)}{e(R)}
\]
and the result follows.
\end{proof}

\begin{corollary}
    \label{co:CM-version}
For $R$ as in Theorem~\ref{thm:LC}, and any nonzero $R$-module $M$ of finite projective dimension, one has
\[
\length_RM \geq e(R) \max\left\{\beta_0(M),  \frac{\beta_1(M)}{2}, \beta_2(M)\right\}.
\]
\end{corollary}

\begin{proof}
We can assume $\length_RM$ is finite, and then the ring $R$ has to be Cohen-Macaulay, and the projective dimension of $M$ equals $\dim R$. Now one can apply Theorem~\ref{thm:LC}(2) to the minimal free resolution of $M$. 
\end{proof}

\section{The non-equidimensional case}
\label{sec:mixeddim}
In this section we extend our results in the previous section to local rings of dimension two that are not equidimensional, by passing to quotients that are. We start with some generalities: For a local ring $R$, following \cite[Definition 2.1]{Hochster/Huneke:1994}, let $j(R)$  be the largest ideal whose dimension, as an $R$-module, is strictly less than $\dim R$. Equivalently, if
\[
0 = P_1 \cap \cdots \cap P_m \cap Q_1 \cap \cdots \cap Q_n
\]
a primary decomposition of $0$ in  $R$, ordered so that the prime ideals $\fp_i\coloneq \sqrt{P_i}$ and $\fq_j\coloneq \sqrt{Q_j}$ satisfy
\[
\dim(R/\fp_i) = \dim R\quad \text{and}\quad \dim(R/\fq_j) < \dim R\\,
\]
then $j(R) = P_1 \cap \cdots \cap P_m$. Since the images of the $P_i$'s give a primary decomposition of $0$ in $R/j(R)$,
we have that $R/j(R)$ is equidimensional, with $\dim R/j(R) = \dim R$, and has no embedded primes. Moreover, since $\dim j(R) < \dim R$, we have $e(R/j(R)) = e(R)$. In fact, $R/j(R)$ is the smallest quotient of $R$ having these
properties. Note also that for any prime $\fp \in \spec R \setminus V$ where $V = V(\fq_1)\cup \cdots \cup V(\fq_n)$, the natural quotient map $R\to R/j(R)$ induces an isomorphism  $R_\fp \cong (R/j(R))_\fp$.

\begin{lemma} 
\label{lem:jR} 
For any local ring $R$, there is a bijection between the isomorphism classes of maximal Cohen-Macaulay (respectively, Ulrich) modules on $R$ and maximal Cohen-Macaulay (respectively, Ulrich) modules on $R/j(R)$ given by restriction of  scalars.
     \end{lemma}

     \begin{proof} 
     Set $J = j(R)$ and $\bar R = R/J$.
       Every maximal Cohen-Macaulay module over $\bar R$ is also maximal Cohen-Macaulay over $R$ since
      $\dim (\bar R) = \dim R$.  For any $\bar R$-module $U$, it is clear from definitions that $\nu_R(U) = \nu_{\bar R}(U)$ and
      $e_R(U) = e_{\bar R}(U)$. In particular, if $U$ an Ulrich
      $\bar R$-module then it is   also Ulrich as an $R$-module.
      
      It remains to show that each maximal Cohen-Macaulay $R$-module $M$ satisfies $JM = 0$.  Since
     $JM\subset M$ the associated primes of $JM$ are subset of those of $M$, and hence a subset of $\{\fp_1,\dots,\fp_m\}$, as $M$ is maximal Cohen-Macaulay. But $J_{\fp_i}=0$ for each $i$ since $\dim J < \dim R$, and thus $JM=0$.
\end{proof}

\begin{lemma} 
\label{le:vtl}
  Assume $R$ is universally catenary local ring (for example, a complete local ring) of dimension two and let $\bar R = R/j(R)$. 
  Set $E = \proj \agr_\fm(R)$ and $\bar E = \proj \agr_{\bar \fm} \bar R$.
  Let $E = E_1 \amalg E_0$ where $E_i$ is the union of the connected components of $E$ having dimension $i$.  
If $E_1$ is reduced     then $\bar E = E_1$.
\end{lemma}

\begin{proof}
Let   $S\coloneq \agr_\fm(R)$ and $\bar S\coloneq \agr_{\bar \fm} \bar R$.  Since $R$ is universally catenary, so is $\bar R$, and so since $\bar R$ is equidimensional, 
we have that  $\bar S$ is also equidimensional; see \cite[Proposition~5.4.8]{Huneke/Swanson:2006}. So, $\bar E$ has no connected components of dimension
$0$; that is, $\bar E$ is a closed subscheme of $E_1$. 

Suppose $E_1$ is reduced.
Since $R$ and $\bar R$ have the same dimension and  multiplicity, the Hilbert polynomials of the graded rings $S$ and $\bar S$
have the same degree and leading coefficient. It follows that the kernel $I$ of the canonical surjection $S \twoheadrightarrow \bar S$  satisfies $\dim I < \dim S = \dim \bar S = 2$. Sheafifying, this gives that the kernel $\mci$ of $\mco_{E} \twoheadrightarrow \mco_{\bar E}$ is given by an ideal of dimension $0$ on each open affine subset. 
  Let $\spec A = U$ be any affine open subset of $E_1$. Since $A$ is reduced $\dim I = \dim A$ for all non-zero ideals
  $I$. Thus $\mci(U) = 0$ and $U \cap \bar E = U$ for all such $U$. 
\end{proof}

The following generalizes Theorem \ref{thm:curves}:

\begin{theorem} 
\label{thm:curves2}
  Let $R$ be a complete local ring of dimension two such that each connected component of $E = \proj \agr_\fm(R)$ of dimension
  one is geometrically reduced.  Then $R$ admits an Ulrich module that is locally free of constant rank on $\spec R \setminus V$ where
  $V = \{\fm\} \cup \{\fq \mid \text{ $\fq$ is minimal with $\dim R/\fq = 1$}\}$.
  \end{theorem}

  \begin{proof} 
    Let $\bar R = R/j(R)$, set $E = \proj \agr_{\fm}(R)$, and let $E_1$ be the union of the connected components of $E$ having dimension one. By assumption, $E_1$ is geometrically reduced, and so in particular it is reduced, and thus by Lemma \ref{le:vtl}, $\bar E = E_1$. It follows that $\bar E$ is geometrically reduced.  Theorem \ref{thm:curves} thus applies to give the existence of an Ulrich module $U$ on $\bar R$ that is locally free of constant rank on the punctured spectrum. The result now follows from Lemma \ref{lem:jR} and the fact that $R_\fp \cong \bar R_\fp$ for all primes $\fp$ not in $V$.
  \end{proof}

  We next generalize Theorem \ref{thm:LC} to allow for rings that may not be equidimensional. A straightforward generalization is not available, since $\chi(F)$ can  be negative for $F$ as in \emph{op.~cit.} if the ring $R$ is not equidimensional; here is a simple example.

  \begin{example} 
  Let $k$ be a field and $R = k[[x,y,z]]/(xz, yz)$. This ring is not equidimensional, as its minimal primes are $\fp = (x,y)$ and $\fq = (z)$ with $\dim(R/\fp) =1$ and  $\dim(R/\fq) =2$. Fix an integer $n\ge 1$ and let $F$ be the finite free complex
\[
0 \longrightarrow R
\xrightarrow{\begin{bmatrix} -y \\ x-z^n \\ \end{bmatrix}} R^2  \xrightarrow{\begin{bmatrix} x + z &  y \end{bmatrix}} R \longrightarrow 0\,.
\]
Writing $K(\bs{x};M)$ for the Koszul complex on a sequence $\bs{x}$ with coefficients in an $R$-module $M$, one has isomorphisms of complexes
\[
F/\fp F \cong K(z;k[[z]]) \oplus \Sigma K(z^n;k[[z]])\quad\text{and}\quad
F/\fq F \cong K(x,y; k[[x,y]])\,.
\]
Thus $F$ has finite length homology, with $\chi(F/\fp F) = 1-n$ and $\chi(F/\fq F) = 1$. Since the cokernel $C$ of the canonical injection $R \hookrightarrow R/\fp \oplus R/\fq$ has finite length, one has $\chi(F \otimes_R C) = 0$ and hence 
\[
\chi(F) = \chi(F/\fp F) + \chi(F/\fq F) = 2-n
\]
which is negative for $n \ge 3$.
\end{example}

The correct invariant to use for rings that need not be equidimensional is 
 the \emph{Dutta multiplicity}, $\chi_\infty(F)$, whose definition is recalled below. In the previous example, $\chi_\infty(F) = 1$.

Let $R$ be a local ring that can be written as the quotient of a regular local ring; this holds, for example, when $R$ is a complete local ring. The rationalized Grothendieck group of $R$ admits a \emph{weight decomposition} 
\[
G_0(R)_\bbQ = \bigoplus_{i=0}^{\dim R} G_0(R)_{(i)}
\]
arising from the Riemann-Roch isomorphism
\[
\tau\colon \mathrm{CH}_*(R)_\bbQ \xrightarrow{\cong} G_0(R)_\bbQ
\]
by setting $G_0(R)_{(i)} = \tau( \mathrm{CH}_i(R)_\bbQ)$; see \cite{Kurano:2004}. Given a finitely generated $R$-module $M$, one can then decompose its class in $G_0(R)$ as 
\[
[M] = \sum_i [M]_{(i)}\,, \quad \text{with $[M]_{(i)}$ in $G_0(R)_{(i)}$.}
\]
We call $M_{(i)}$ the \emph{weight $i$ component} of $[M]$.

Recall that $E(R)_\bbQ$ is the quotient of the rationalized Grothendieck group $G_0(R)_\bbQ$ by classes that are numerically equivalent to zero.  The weight decomposition descends to $E(R)_\bbQ$, giving a decomposition
\[
E(R)_\bbQ  = \bigoplus_{i=0}^{\dim R} E(R)_{(i)}.
\]
These decompositions enjoy the following properties; see \cite[Proposition~3.7]{Kurano:2004} for proofs. We write $\minspec R$ for the set
of minimal prime ideals of $R$. 

\begin{proposition}  
\label{prop:wd}
  Let $R$ be local ring that is a homomorphic image of a regular local ring. Then:
\begin{enumerate}[\quad\rm(1)]
\item For a finitely generated $R$-module $M$, we have $[M]_{(i)} = 0$ for $i > \dim_RM$.
\item When $\chr(R) = p>0$ and $R$ is $F$-finite with perfect residue field, 
  $G_0(R)_{(i)}$ is the eigenspace of eigenvalue $p^i$ for the operator on $G_0(R)_\bbQ$ induced by restriction of scalars along the Frobenius endomorphism.
\item With $V = \{\fp \in \spec R \mid \dim(R/\fp) =\dim R\}$,  the map 
\[
E(R)_{(\dim R)} \xrightarrow{\cong} \bigoplus_{\fp \in V} \bbQ
\]
sending $M$ to $(\length_{R_\fp}M_\fp)_{\fp \in V}$ is an isomorphism.
\item 
When $R$ is equidimensional of dimension two,  $E(R)_{(0)} =0= E(R)_{(1)}$.
\end{enumerate}
\end{proposition}

\begin{definition} 
  Let $R$ be a complete local ring $R$. The \emph{Dutta multiplicity} of a finite free $R$-complex $F$ having finite length homology
  is
  \[
\chi^R_\infty(F) =  \chi(F, [R]_{(\dim R)}).
  \]
  For an arbitrary local ring $R$, we define
  \[
  \chi^R_\infty(F) =  \chi^{\hat{R}}_\infty(\hat{R} \otimes_R F).
  \]
  \end{definition}

  \begin{remark}
When $R$  is a complete local ring of positive characteristic $p$ with perfect residue field, one has
\[
    \chi_\infty(F) = \lim_{e \to \infty}\frac{\chi(\phi^e F)}{p^{de}}
  \]
 where $\phi^e$ is the $e$-th iterate of extension of scalars along the Frobenius; for a proof, see, for instance,  \cite[pp.~429]{Roberts:1989}.
  \end{remark}

  The following is an immediate consequence of part (4) of Proposition \ref{prop:wd}. 

  \begin{corollary}  
  \label{cor:chi}
    If $R$ has dimension two and its completion is equidimensional, then 
    $\chi_\infty(F) = \chi(F)$. \qed
\end{corollary}

\begin{theorem}  
\label{thm:LC2}
Let $(R, \fm)$ be a local ring of dimension two. If the connected components of $\proj(\agr_\fm(R))$ of dimension one are geometrically reduced, then for any finite free $R$-complex $F= 0 \to F_2 \to F_1 \to F_0 \to  0$ that is minimal and has
non-zero, finite length homology, one has 
\[
\chi_\infty(F) \geq e(R)\cdot \max\left\{\rank F_0,  \frac{\rank F_1 }{2}, \rank F_2 \right\}.
\]
\end{theorem}

\begin{proof} 
As in the proof of Theorem~\ref{thm:LC}, we may assume $R$ is complete with infinite residue field. Let $\bar R = R/j(R)$, a complete, equidimensional local ring of dimension $2$ with  $e(R) = e(\bar R)$. As $\dim j(R) < 2$,
  we have $[R]_{(2)} = [\bar R]_{(2)}$ in $G_0(R)_\bbQ$ and thus
  \[
  \chi_\infty(F) = \chi_\infty( \bar R \otimes_R F) = \chi( \bar R \otimes_R F)\,.
  \]
  The second equality coming from Corollary \ref{cor:chi}. As in the proof of Theorem \ref{thm:curves2}, it follows from Lemma \ref{le:vtl} that $\proj (\agr_{\bar \fm} \bar R)$ is geometrically reduced. The inequality thus follows from Theorem \ref{thm:LC} applied to $\bar R$ and $F \otimes_R \bar R$. 
\end{proof}

\begin{ack}
The authors were partly supported by National Science Foundation  (NSF) grants DMS-2001368 and DMS-2502004 (SBI);  DMS-2302430 and DMS-2424441 (LM); DMS-2200732 and DMS-2502519 (MEW).

This material is partly based upon work supported by the NSF under Grant No. DMS-1928930 and by the Alfred P. Sloan Foundation under grant G-2021-16778, while the authors were in residence at the Simons Laufer Mathematical Sciences Institute (formerly MSRI) in Berkeley, California, during the Spring 2024 semester.  
\end{ack}

\bibliographystyle{amsplain}
\providecommand{\bysame}{\leavevmode\hbox to3em{\hrulefill}\thinspace}
\providecommand{\MR}{\relax\ifhmode\unskip\space\fi MR }
\providecommand{\MRhref}[2]{%
  \href{http://www.ams.org/mathscinet-getitem?mr=#1}{#2}
}
\providecommand{\href}[2]{#2}

\end{document}